\documentclass{amsart}
\usepackage{amssymb}
\usepackage{amsfonts}

\newtheorem{theorem}{Theorem}
\theoremstyle{plain}

\newtheorem{definition}{Definition}
\newtheorem{example}{Example}

\newtheorem{lemma}{Lemma}

\newtheorem{remark}{Remark}

\numberwithin{equation}{section}
\input{tcilatex}

\begin{document}
\title[ Fixed points of $\left( \lambda ,\rho \right) $-firmly nonexpansive
mappings]{Approximating fixed points of $\left( \lambda ,\rho \right) $%
-firmly nonexpansive mappings in modular function spaces}
\author{Safeer Hussain Khan}
\address{Safeer Hussain Khan, Department of Mathematics, Stististics and
Physics, Qatar University, Doha 2713, Qatar}
\email{safeer@qu.edu.qa ; safeerhussain5@yahoo.com }

\begin{abstract}
In this paper, we first introduce an iterative process in modular function
spaces and then extend the idea of a $\lambda $-firmly nonexpansive mapping
from Banach spaces to modular function spaces. We call such mappings as $%
\left( \lambda ,\rho \right) $-firmly nonexpansive mappings. We incorporate
the two ideas to approximate fixed points of $\left( \lambda ,\rho \right) $%
-firmly nonexpansive mappings using the above mentioned iterative process in
modular function spaces. We give an example to validate our results.
\end{abstract}

\maketitle

\section{Introduction}

Fixed point theory has several applications in different disciplines and
therefore it has been a flourishing area of research. The metric fixed pint
theory in the framework of Banach spaces usually involves a close link of
geometric and topological conditions. Fixed point theory in modular function
spaces and metric fixed point theory are near relatives because former
provides modular equivalents of norm and metric concepts. Modular spaces are
extensions of the classical Lebesgue and Orlicz spaces, and in many
instances conditions cast in this framework are more natural and more easily
verified than their metric analogs. For more discussion, see for example,
Khamsi and Kozlowski \cite{MAKhamsi}.

Nowadays, a vigorous research activity is developed in the area of numerical
reckoning fixed points for suitable classes of nonlinear operators: see, for
example, \cite{TTP1, TTP2} , and applications to image recovery and
variational inequalities: see \cite{YPLY, YAPL, YLP, YLPZ}. Existence of
fixed points in\ modular function spaces has been studied by many
researchers, for example, Khamsi and Kozlowski \cite{MAKhamsi} and the
references therein. Dhompongsa et al. \cite{DBKP} have proved the existence
of fixed point of $\rho $-contractions under certain conditions. Buthina and
Kozlowski \cite{BK}, for the first time, proved results on approximating
fixed points in modular function spaces through Mann and Ishikawa iterative
processes. Some work for multivalued mappings in modular function spaces
using Mann iterative process was done by Khan and Abbas \cite{KA}. Khan \cite%
{safeer} introduced an iterative process for approximation of fixed points
of certain mappings in Banach spaces. This process is independent of both
Mann and Ishikawa iterative processes in the sense that neither reduces to
the other under the given conditions. Moreover, it is faster than all of
Picard, Mann and Ishikawa iterative processes in case of contractions \cite%
{safeer}. We extend this process to the framework of modular function
spaces. On the other hand, $\lambda $-firmly nonexpansive mappings in Banach
spaces{\footnotesize \ }have attracted many researchers. For a discussion on
such mappings, see for example Ruiz et al. \cite{Ruiz} and the references
cited therein. As far as we know, no work has been done until now on this
kind of mappings in modular function spaces. We thus introduce the idea of
the so-called $\left( \lambda ,\rho \right) $-firmly nonexpansive mappings,
in short $\left( \lambda ,\rho \right) $-FNEM. We approximate the fixed
points of such mappings using the above mentioned iterative process in
modular function spaces. This will create new results in modular function
spaces.

\section{Preliminaries}

Here is a brief note on modular function spaces to make the discussion
self-contained. This has mainly been extracted from Khamsi and Kozlowski 
\cite{MAKhamsi}.

Let $\Omega $ be a nonempty set and $\Sigma $ a nontrivial $\sigma $-algebra
of subsets of $\Omega .$ Let $\mathcal{P}$ be a $\delta $-ring of subsets of 
$\Omega ,$ such that $E\cap A\in \mathcal{P}$ for any $E\in \mathcal{P}$ and 
$A\in $ $\Sigma .$ Let us assume that there exists an increasing sequence of
sets $K_{n}\in \mathcal{P}$ such that $\Omega =\cup K_{n}$ (for instance, $%
\mathcal{P}$ can be the class of sets of finite measure in a $\sigma $%
-finite measure space). By $1_{A},$ we denote the characteristic function of
the set $A$ in $\Omega .$ By $\mathcal{E}$ we denote the linear space of all
simple functions with supports from $\mathcal{P}$. By $\mathcal{M}_{\infty }$
we will denote the space of all extended measurable functions, i.e., all
functions $f:\Omega \rightarrow \lbrack -\infty ,\infty ]$ such that there
exists a sequence $\{g_{n}\}\subset \mathcal{E},$ $\left\vert
g_{n}\right\vert \leq \left\vert f\right\vert $ and $g_{n}(\omega
)\rightarrow f(\omega )$ for all $\omega \in \Omega .$

\begin{definition}
Let $\rho :\mathcal{M}_{\infty }\rightarrow \lbrack 0,\infty ]$ be a
nontrivial, convex and even function. We say that $\rho $ is a regular
convex function pseudomodular if

\begin{enumerate}
\item $\rho (0)=0;$

\item $\rho $ is monotone, i.e., $\left\vert f(\omega )\right\vert \leq
\left\vert g(\omega )\right\vert $ for any $\omega \in \Omega $ implies $%
\rho (f)\leq \rho (g),$ where $f,g\in \mathcal{M}_{\infty };$

\item $\rho $ is orthogonally subadditive, i.e., $\rho (f1_{A\cup B})\leq
\rho (f1_{A})+\rho (f1_{B})$ for any $A,B\in \Sigma $ such that $A\cap B\neq
\phi ,$ $f\in \mathcal{M}_{\infty };$

\item $\rho $ has Fatou property, i.e., $\left\vert f_{n}(\omega
)\right\vert \uparrow \left\vert f(\omega )\right\vert $ for all $\omega \in
\Omega $ implies $\rho (f_{n})\uparrow \rho (f),$ where $f\in \mathcal{M}%
_{\infty };$

\item $\rho $ is order continuous in $\mathcal{E},$ i.e., $g_{n}\in \mathcal{%
E},$ and\ $\ \ \left\vert g_{n}(\omega )\right\vert $~$\downarrow 0$ implies 
$\rho (g_{n})\downarrow 0.$
\end{enumerate}
\end{definition}

A set $A\in $ $\Sigma $ is said to be $\rho $-null if $\rho (g1_{A})=0$ for
every $g\in \mathcal{E}.$ A property $p(\omega )$ is said to hold $\rho $%
-almost everywhere ($\rho $-a.e.) if the set $\{\omega \in \Omega :p(\omega
) $ does not hold$\}$ is $\rho $-null. As usual,we identify any pair of
measurable sets whose symmetric difference is $\rho $-null as well as any
pair of measurable functions differing only on a $\rho $-null set. With this
in mind we define%
\begin{equation*}
\mathcal{M}\left( \Omega ,\Sigma ,\mathcal{P},\rho \right) =\left\{ f\in 
\mathcal{M}_{\infty }:\left\vert f(\omega )\right\vert <\infty ~\rho \text{%
-a.e.}\right\} ,
\end{equation*}%
where $f\in \mathcal{M}\left( \Omega ,\Sigma ,\mathcal{P},\rho \right) $ is
actually an equivalence class of functions equal $\rho $-a.e. rather than an
individual function. Where no confusion exists, we will write $\mathcal{M}$
instead of $\mathcal{M(}\Omega ,\Sigma ,\mathcal{P},\rho ).$

It is easy to see that $\rho :$ $\mathcal{M\rightarrow }[0,\infty ]$
posseses the following properties:

$1.$ $\rho (0)=0$ iff $f=0\ \rho $-a.e.

$2.$ $\rho (\alpha f)=\rho (f)$ for every scalar $\alpha $ with $\
\left\vert \alpha \right\vert =1$ and $f\in \mathcal{M}.$

$3.$ $\rho (\alpha f+\beta g)\leq \rho (f)+\rho (g)$ if $\alpha +\beta =1,$ $%
\alpha ,\beta \geq 0~$and $f,g\in \mathcal{M}.$

$\rho $ is\ called a convex modular if, in addition, the following property
is satisfied:

$3^{\prime }.$ $\rho (\alpha f+\beta g)\leq \alpha \rho (f)+\beta \rho (g)$
if $\alpha +\beta =1,$ $\alpha ,\beta \geq 0~$and $f,g\in \mathcal{M}.$

\begin{definition}
Let $\rho $ be a regular function pseudomodular. We say that $\rho $ is a
regular convex function modular if $\rho (f)=0$ implies $f=0\ \rho $-a.e.
\end{definition}

The class of all nonzero regular convex function modulars defined on $%
\Omega
$ is denoted by $\Re .$

\noindent The convex function modular $\rho $ defines the modular function
space $L_{\rho }$ as 
\begin{equation*}
L_{\rho }=\{f\in \mathcal{M}:\text{ }\rho (\lambda f)\rightarrow 0\ \text{as}%
\ \lambda \rightarrow 0\}.
\end{equation*}

\noindent Generally, the modular $\rho $ is not sub-additive and therefore
does not behave as a norm or a distance. However, the modular space $L_{\rho
}$ can be equipped with an $F$-norm defined by 
\begin{equation*}
\Vert f\Vert _{\rho }=\inf \{\alpha >0:\rho \left( \frac{f}{\alpha }\right)
\leq \alpha \}.
\end{equation*}%
In case $\rho $ is convex modular, 
\begin{equation*}
\Vert f\Vert _{\rho }=\inf \{\alpha >0:\rho \left( \frac{f}{\alpha }\right)
\leq 1\}
\end{equation*}%
defines a norm on the modular space $L_{\rho },$ and is called the Luxemburg
norm.

Define $L_{\rho }^{0}=\left\{ f\in L_{\rho }:\rho \left( f,.\right) \text{
is order continuous}\right\} $\ and the linear space\emph{\ }$E_{\rho
}=\left\{ f\in L_{\rho }:\lambda f\in L_{\rho }^{0}\text{ for every }\lambda
>0\right\} .$

\begin{definition}
$\rho \in \Re $ is said to satisfy the $\Delta _{2}$-condition, if \ $%
\sup_{n\geq 1}\rho (2f_{n},D_{k})\rightarrow 0$ as $k\rightarrow \infty $
whenever $\{D_{k}\}$ decreases to $\phi $ and $\sup_{n\geq 1}\rho
(f_{n},D_{k})\rightarrow 0$ as $k\rightarrow \infty .$
\end{definition}

If $\rho $\ is convex and satisfies the $\Delta _{2}$-condition, then$\
L_{\rho }=E_{\rho }.$ Moreover, $\rho $ satisfies the $\Delta _{2}$%
-condition if and only if $F$-norm convergence and modular convergence are
equivalent.

\begin{definition}
Let $\rho \in \Re .$

$(i)$ Let $r>0,~\varepsilon >0.$ Define 
\begin{equation*}
D_{1}(r,\epsilon )=\left\{ \left( f,g\right) :f,g\in L_{\rho },\rho (f)\leq
r,\rho (g)\leq r,\rho (f-g)\geq \varepsilon r\right\} .
\end{equation*}%
Let 
\begin{equation*}
\delta _{1}(r,\epsilon )=\inf \left\{ 1-\frac{1}{r}\rho (\frac{f+g}{2}%
):(f,g)\in D_{1}(r,\epsilon )\right\} \text{ if}\ \text{\ }D_{1}(r,\epsilon
)\neq \phi ,
\end{equation*}%
and $\delta _{1}(r,\epsilon )=1$ if $\ D_{1}(r,\epsilon )=\phi .$ We say
that $\rho $ satisfies $(UC1)$ if for every $r>0,\epsilon >0,$ $\delta
_{1}(r,\epsilon )>0.$ Note, that for every $r>0,D_{1}(r,\epsilon )\neq \phi
, $ for $\epsilon >0$ small enough.

$(ii)$ We say that $\rho $ satisfies $(UUC1)$ if for every $s\geq 0,\epsilon
>0,$ there exists $\eta _{1}(s,\epsilon )>0$ depending only upon $s$ and $%
\epsilon $ such that $\delta _{1}(r,\epsilon )>\eta _{1}(s,\epsilon )>0$ for
any $r>s.$
\end{definition}

Note that $(UC1)$ implies $(UUC1).$

\begin{definition}
\label{1.3} Let $\rho \in \Re .$ The sequence $\{f_{n}\}\subset L_{\rho }$
is called:

\begin{itemize}
\item $\rho $-convergent to $f\in L_{\rho }$ if $\rho (f_{n}-f)\rightarrow 0$
as $n$ $\rightarrow \infty .$

\item $\rho $-Cauchy, if $\rho (f_{n}-f_{m})\rightarrow 0$ as $n$ and $%
m\rightarrow \infty .$
\end{itemize}
\end{definition}

Note that, $\rho $-convergence does not imply $\rho $-Cauchy since $\rho $
does not satisfy the triangle inequality. In fact, one can show that this
will happen if and only if $\rho $ satisfies the $\Delta _{2}$-condition.

\begin{definition}
Let $\rho \in \Re .$ A subset $D\subset L_{\rho }$ is called
\end{definition}

\begin{itemize}
\item $\rho $-closed if the $\rho $-limit of a $\rho $-convergent sequence
of $D$ always belongs to $D.$

\item $\rho $-a.e. closed if the $\rho $-a.e. limit of a $\rho $-a.e.
convergent sequence of $D$ always belongs to $D.$

\item $\rho $-compact if every sequence in $D$ has a $\rho $-convergent
subsequence in $D.$

\item $\rho $-a.e. compact if every sequence in $D$ has a $\rho $-a.e.
convergent subsequence in $D.$

\item $\rho $-bounded if $diam_{\rho }(D)=\sup \{\rho (f-g):f,g\in
D\}<\infty .$
\end{itemize}

A sequence $\{t_{n}\}\subset (0,1)$ is called bounded away from $0$ if there
exists $a>0$ such that $t_{n}\geq a$ for every $n\in 
\mathbb{N}
.$ Similarly, $\{t_{n}\}\subset (0,1)$ is called bounded away from $1$ if
there exists $b<1$ such that $t_{n}\leq b$ for every $n\in 
\mathbb{N}
.$ The following lemma can be seen as an analogue of a famous lemma due to
Schu \cite{Schu} in Banach spaces.

\begin{lemma}
\cite[Lemma 4.1]{MAKhamsi}\label{Schu Type} \ Let $\rho \in \Re $ satisfy $%
(UUC1)$ and let $\{t_{k}\}\subset (0,1)$ be bounded away from $0$ and $1.$
If there exists $R>0$ such that%
\begin{equation*}
\limsup_{n\rightarrow \infty }\rho (f_{n})\leq R,\text{ }\limsup_{n%
\rightarrow \infty }\rho (g_{n})\leq R,
\end{equation*}%
and 
\begin{equation*}
\lim_{n\rightarrow \infty }\rho (t_{n}f_{n}+(1-t_{n})g_{n})=R,
\end{equation*}%
then%
\begin{equation*}
\lim_{n\rightarrow \infty }\rho (f_{n}-g_{n})=0.
\end{equation*}
\end{lemma}

\noindent A function $f\in L_{\rho }$ is called a fixed point of $T:L_{\rho
}\rightarrow L_{\rho }$ if $f=Tf.$ The set of all fixed points of $T$ will
be denoted by $F_{\rho }(T).$

The $\rho $-distance from an $f\in L_{\rho }$ to a set $D\subset L_{\rho }$
is given as follows: 
\begin{equation*}
dist_{\rho }(f,D)=\inf \{\rho (f-h):h\in D\}.
\end{equation*}

The following definition is a modular space version of the condition $(I)$
of Senter and Dotson \cite{SD}. Let $D\subset L_{\rho }.$ A mapping $%
T:D\rightarrow D$ is said to satisfy condition $(I)$ if there exists a
nondecreasing function $\ell :[0,\infty )\rightarrow \lbrack 0,\infty )$
with $\ell (0)=0,\;\ell (r)>0$ for all $r\in (0,\infty )$ such that 
\begin{equation*}
\rho (f-Tf)\geq \ell (dist_{\rho }(f,F_{\rho }(T))
\end{equation*}%
for all $f\in D.$

\begin{definition}
A mapping $T:D\rightarrow D$ is called $\rho $-nonexpansive mapping if 
\begin{equation*}
\rho (Tf-Tg)\leq \rho \left( f-g\right) \text{ for all }f,g\in D\text{ }.
\end{equation*}
\end{definition}

The folowing general theorem (\cite[Theorem 5.7]{MAKhamsi}) confirms the
existence fixed points of $\rho $-nonexpansive mappings.

\begin{theorem}
\label{Khamsi}Assume $\rho \in \Re \ $satisfy $(UUC1)$. Let $D$ be a $\rho $%
-closed, $\rho $-bounded convex and nonempty subset of $L_{\rho }.$ Then,
any $T:D\rightarrow D$ pointwise asymptotically nonexpansive mapping has a
fixed point. Moreover, the set of all fixed points $F(T)$ is $\rho $-closed
and convex.
\end{theorem}

\section{Fixed points approximation of $\left( \protect\lambda ,\protect\rho %
\right) $-FNEM}

We first extend the idea of a $\lambda $-firmly nonexpansive mapping from
Banach spaces to modular function spaces and call it $\left( \lambda ,\rho
\right) $-firmly nonexpansive mapping. We define the idea as follows.

\begin{definition}
Let $D\subset L_{\rho }.$ We say that a mapping $T:D\rightarrow D$ is called 
$\left( \lambda ,\rho \right) $-firmly nonexpansive mapping if for given $%
\lambda \in (0,1),$ 
\begin{equation*}
\rho (Tf-Tg)\leq \rho \left[ (1-\lambda )\left( f-g\right) +\lambda (Tf-Tg)%
\right] \text{ for all }f,g\in D\text{ }.
\end{equation*}
\end{definition}

For simplicity, we denote a $\left( \lambda ,\rho \right) $-firmly
nonexpansive mapping by $\left( \lambda ,\rho \right) $-FNEM.

\begin{lemma}
\label{FNtoN}$\left( \lambda ,\rho \right) $-firmly nonexpansivess implies $%
\rho $-nonexpansiveness.
\end{lemma}

\begin{proof}
Let $T:D\rightarrow D$ be $\left( \lambda ,\rho \right) $-firmly
nonexpansive mapping, then 
\begin{eqnarray*}
\rho (Tf-Tg) &\leq &\rho \left[ (1-\lambda )\left( f-g\right) +\lambda
(Tf-Tg)\right] \\
&\leq &(1-\lambda )\rho \left( f-g\right) +\lambda \rho (Tf-Tg)
\end{eqnarray*}%
for all $f,g\in D.$ This implies that $(1-\lambda )\rho (Tf-Tg)\leq
(1-\lambda )\rho \left( f-g\right) $ and hence $\rho (Tf-Tg)\leq \rho \left(
f-g\right) $ as $\lambda \neq 1.$
\end{proof}

\begin{lemma}
\label{ExistCC}The set of \ fixed points $F_{\rho }(T)$ of a $\left( \lambda
,\rho \right) $-firmly nonexpansive mapping is nonempty. Moreover, it is $%
\rho $-closed and convex.
\end{lemma}

\begin{proof}
It follows from Lemma \ref{FNtoN} and Theorem \ref{Khamsi}.
\end{proof}

Next we introduce the following iterative process in the setting of modular
function spaces. For a mapping $T:D\rightarrow D,$ we\ define a sequence $%
\left\{ f_{n}\right\} $ by the following iterative process:

\begin{eqnarray}
f_{1} &\in &D,  \label{Def} \\
f_{n+1} &=&Tg_{n},  \notag \\
g_{n} &=&(1-\alpha _{n})f_{n}+\alpha _{n}Tf_{n},~n\in \mathbb{N}  \notag
\end{eqnarray}%
where $\left\{ \alpha _{n}\right\} \subset (0,1)$ is bounded away from both $%
0$ and $1.$

For details on a similar iterative process but in Banach spaces, see \cite%
{safeer}.

In this paper, using the above two ideas together, we prove our main result
for approximating fixed points in modular function spaces. We give a simple
numerical example to support and validate our results.

We are now in a position to give our main results as follows.

\begin{theorem}
\textit{\label{thm}}Let $\rho \in \Re \ $satisfy $(UUC1)$ and $\Delta _{2}$%
-condition. Let $D$ be a nonempty $\rho $-closed, $\rho $-bounded and convex
subset of $L_{\rho }.$ Let $T:D\rightarrow D~$be a $\left( \lambda ,\rho
\right) $-FNEM. Let $\left\{ f_{n}\right\} \subset D$ be defined\ by the
iterative process: Then 
\begin{equation*}
\lim_{n\rightarrow \infty }\rho (f_{n}-w)\text{ exists for all }w\in F_{\rho
}(T),
\end{equation*}%
and%
\begin{equation*}
\lim_{n\rightarrow \infty }\rho (f_{n}-Tf_{n})=0.
\end{equation*}%
\textbf{\ }
\end{theorem}

\begin{proof}
Let $w\in F_{\rho }(T).$ To prove that $\lim_{n\rightarrow \infty }\rho
(f_{n}-w)$ exists for all $w\in F_{\rho }(T),$ consider 
\begin{eqnarray*}
\rho (f_{n+1}-w) &=&\rho \left( Tg_{n}-Tw\right) \\
&\leq &\rho \left[ (1-\lambda )\left( g_{n}-w\right) +\lambda \left(
Tg_{n}-Tw\right) \right] \\
&\leq &(1-\lambda )\rho \left( g_{n}-w\right) +\lambda \rho \left(
Tg_{n}-Tw\right) \text{ \ by convexity of }\rho .
\end{eqnarray*}

This implies $\rho \left( Tg_{n}-Tw\right) \leq \rho \left( g_{n}-w\right) $
and hence 
\begin{equation}
\rho (f_{n+1}-w)\leq \rho \left( g_{n}-w\right) .  \label{gf}
\end{equation}%
Also, because $T$ is a $\left( \lambda ,\rho \right) $-FNEM, 
\begin{equation*}
\rho \left( Tf_{n}-Tw\right) \leq (1-\lambda )\rho \left( f_{n}-w\right)
+\lambda \rho \left( Tf_{n}-Tw\right)
\end{equation*}%
implies $\rho \left( Tf_{n}-Tw\right) \leq \rho \left( f_{n}-w\right) ,$
therefore 
\begin{eqnarray*}
\rho (f_{n+1}-w) &\leq &\rho \left( g_{n}-w\right) \\
&=&\rho \lbrack (1-\alpha _{n})\rho \left( f_{n}-w\right) +\alpha _{n}\rho
\left( Tf_{n}-Tw\right) ] \\
&\leq &(1-\alpha _{n})\rho \left( f_{n}-w\right) +\alpha _{n}\rho \left(
Tf_{n}-Tw\right) \\
&\leq &(1-\alpha _{n})\rho \left( f_{n}-w\right) +\alpha _{n}\rho \left(
f_{n}-w\right) \\
&=&\rho \left( f_{n}-w\right) .
\end{eqnarray*}

Thus $\lim_{n\rightarrow \infty }\rho (f_{n}-w)$ exists for each $w\in
F_{\rho }(T).$

Suppose that 
\begin{equation}
\lim_{n\rightarrow \infty }\rho (f_{n}-w)\text{ }=m  \label{4}
\end{equation}%
where $m\geq 0.$

Note that the above calculations also give the following inequality:%
\begin{equation}
\rho \left( g_{n}-w\right) \leq \rho \left( f_{n}-w\right) .  \label{fg}
\end{equation}

Next, we prove that $\lim_{n\rightarrow \infty }\rho (f_{n}-Tf_{n})=0.$ Now
using \ref{fg}, \ref{gf} and \ref{4}, we have 
\begin{equation*}
m=\lim_{n\rightarrow \infty }\rho (f_{n}-w)=\lim_{n\rightarrow \infty }\rho
(\ g_{n}-w)\leq \rho \left( f_{n}-w\right) =m.
\end{equation*}%
This gives 
\begin{equation*}
\lim_{n\rightarrow \infty }\rho (\ g_{n}-w)=\ m.
\end{equation*}%
Moreover, 
\begin{equation}
\limsup_{n\rightarrow \infty }\rho (\ Tf_{n}-w)\leq \lim_{n\rightarrow
\infty }\rho (\ f_{n}-w)=\ m.  \label{5}
\end{equation}%
But then $\rho (f_{n+1}-w)\leq \rho (g_{n}-w)$ implies that 
\begin{eqnarray*}
\lim_{n\rightarrow \infty }\rho \left[ (1-\alpha _{n})(f_{n}-w)+\alpha
_{n}(Tf_{n}-w)\right] &=&\lim_{n\rightarrow \infty }\rho \left[ (1-\alpha
_{n})f_{n}+\alpha _{n}Tf_{n})-w\right] \\
&=&\lim_{n\rightarrow \infty }\rho (\ g_{n}-w) \\
&=&\ m.
\end{eqnarray*}%
Now by $\left( \ref{4}\right) ,\left( \ref{5}\right) $ and Lemma \ref{Schu
Type},$~$we have 
\begin{equation*}
\lim_{n\rightarrow \infty }\rho (f_{n}-Tf_{n}\ )=0.
\end{equation*}%
as required.
\end{proof}

Using the above result, we now prove our convergence result for
approximating fixed points of $\left( \lambda ,\rho \right) $-firmly
nonexpansive mappings in modular function spaces using our iterative process 
$\left( \ref{Def}\right) $ as follows.

\begin{theorem}
\textit{\label{thm4} }Let $\rho \in \Re \ $satisfy $(UUC1)$ and $\Delta _{2}$%
-condition. Let $D$ be a nonempty $\rho $-compact and convex subset of $%
L_{\rho }.$ Let $T:D\rightarrow D~$be a $\left( \lambda ,\rho \right) $%
-FNEM. Let $\ \left\{ f_{n}\right\} $ be as defined\ in Theorem \ref{thm}.
Then $\left\{ f_{n}\right\} $ $\rho $-converges to a fixed point of $T.$
\end{theorem}

\begin{proof}
Since $D$ is $\rho $-compact, there exists a subsequence $\left\{
f_{n_{k}}\right\} $ of $\left\{ f_{n}\right\} $ such that $%
\lim_{k\rightarrow \infty }\left( f_{n_{k}}-z\right) =0$ for some $z\in D.$
Since $T$ is a $\left( \lambda ,\rho \right) $-FNEM, using convexity of $%
\rho ,$ we have 
\begin{eqnarray*}
\rho \left( \frac{z-Tz}{3}\right) &=&\rho \left( \frac{z-f_{n_{k}}}{3}+\frac{%
f_{n_{k}}-Tf_{n_{k}}}{3}+\frac{Tf_{n_{k}}-Tz}{3}\right) \\
&\leq &\frac{1}{3}\rho (z-f_{n_{k}})+\frac{1}{3}\rho (f_{n_{k}}-Tf_{n_{k}})+%
\frac{1}{3}\rho (Tf_{n_{k}}-Tz) \\
&\leq &\rho (z-f_{n_{k}})+\rho (f_{n_{k}}-Tf_{n_{k}})+\rho (f_{n_{k}}-z) \\
&\leq &2\rho \left( z-f_{n_{k}}\right) +\rho (f_{n_{k}}-Tf_{n_{k}}).
\end{eqnarray*}

Applying Theorem \ref{thm}, $\lim_{n\rightarrow \infty }\rho
(f_{n_{k}}-Tf_{n_{k}})=0.$ That is, $\rho (\frac{z-Tz}{3})=0.$ Hence $z$ is
a fixed point of $T.$ That is, $\left\{ f_{n}\right\} $ $\rho $-converges to
a fixed point of $T.$
\end{proof}

\begin{theorem}
\label{SC} Let $\rho \in \Re \ $satisfy $(UUC1)~$and $\Delta _{2}$%
-condition. Let $D$ be a nonempty $\rho $-closed$,$ $\rho $-bounded and
convex subset of $L_{\rho }.$ Let $T:D\rightarrow D~$be a $\left( \lambda
,\rho \right) $-FNEM satisfying condition $(I).\ $Let $\left\{ f_{n}\right\} 
$ be\ as defined\ in Theorem \ref{thm}. Then $\left\{ f_{n}\right\} $ $\rho $%
-converges to a fixed point of $T.$
\end{theorem}

\begin{proof}
By Theorem \ref{thm}, $\lim_{n\rightarrow \infty }\rho (f_{n}-w)$ exists for
all $w\in F_{\rho }(T).\ $Suppose that $\lim\limits_{n\rightarrow \infty
}\rho \left( f_{n}-w\right) $ $=m>0$ because otherwise $\lim\limits_{n%
\rightarrow \infty }\rho \left( f_{n}-w\right) $ $=0$ means nothing left to
prove. Now by Theorem \ref{thm}, we have $\rho \left( f_{n+1}-w\right) \leq
\rho \left( f_{n}-w\right) \ $so that 
\begin{equation*}
dist_{\rho }(f_{n+1},F_{\rho }(T))\leq dist_{\rho }(f_{n},F_{\rho }(T)).
\end{equation*}%
This means that $\lim_{n\rightarrow \infty }dist_{\rho }(f_{n},F_{\rho }(T))$
exists. Applying condition$\;(I)$ and Theorem \ref{thm}, we have 
\begin{equation*}
\lim\limits_{n\rightarrow \infty }\ell (dist_{\rho }(f_{n},F_{\rho
}(T)))\leq \lim\limits_{n\rightarrow \infty }\rho (f_{n}-Tf_{n})=0.
\end{equation*}%
Since $\ell $\ is a nondecreasing function and $\ell (0)=0,$ therefore 
\begin{equation}
\lim\limits_{n\rightarrow \infty }dist_{\rho }(f_{n},F_{\rho }(T))=0.
\label{cs}
\end{equation}

To prove that $\{f_{n}\}$ is a $\rho $-Cauchy sequence in $D,$let $%
\varepsilon >0.$ By $\left( \ref{cs}\right) ,$ there exists a constant $%
n_{0} $ such that for all $n\geq n_{0,}$ 
\begin{equation*}
dist_{\rho }(f_{n},F_{\rho }(T))<\frac{\varepsilon }{2}.
\end{equation*}%
Hence there exists a $y$ $\in F_{\rho }(T)$ such that 
\begin{equation*}
\rho \left( f_{n_{0}}-y\right) <\varepsilon .
\end{equation*}%
Now for $m,n\geq n_{0},$ 
\begin{eqnarray*}
\rho \left( \frac{f_{n+m}-f_{n}}{2}\right) &\leq &\frac{1}{2}\rho \left(
f_{n+m}-y\right) +\frac{1}{2}\rho \left( f_{n}-y\right) \\
&\leq &\rho \left( f_{n_{0}}-y\right) \\
&<&\varepsilon .
\end{eqnarray*}%
By $\Delta _{2}$-condition, $\rho \left( f_{n+m}-f_{n}\right) <\varepsilon $
for $m,n\geq n_{0}.$ Hence $\{f_{n}\}$\ is a $\rho $-Cauchy sequence in a $%
\rho $-closed subset $D$ of $L_{\rho },$ and so it converges in $D.$ Let $%
\lim\limits_{n\rightarrow \infty }f_{n}=w.$ Then $dist_{\rho }(w,F_{\rho
}(T))=\lim\limits_{n\rightarrow \infty }dist_{\rho }(f_{n},F_{\rho }(T))=0$
by $\left( \ref{cs}\right) .$ Since by Lemma \ref{ExistCC} $F_{\rho }(T)$ is
closed, $w\in F_{\rho }(T).$That is, $\left\{ f_{n}\right\} $ $\rho $%
-converges to a fixed point of $T.$
\end{proof}

We now give the following example to show the Theorem \ref{SC} is indeed
valid.

\begin{example}
Let the set of real numbers $%
\mathbb{R}
$ be the space modulared as$\rho (f)=\left\vert f\right\vert .$ It follows
that $\rho \in \Re $ satisfies $(UUC1)$ and $\Delta _{2}$-condition. Let $%
D=\left\{ f\in L_{\rho }:1\leq f<\infty \right\} .$ Define $T:D\rightarrow D$
as: 
\begin{equation*}
Tf=\frac{2f+1}{3}.
\end{equation*}%
Obviously $D$ is a $\rho $-compact subset of $%
\mathbb{R}
.~$Note that $F_{\rho }(T)=\{1\}\neq \phi .$ Define a continuous
nondecreasing function $\ell :[0,\infty )\rightarrow \lbrack 0,\infty )$ by $%
\ell (r)=\dfrac{r}{6}.$ We first show that $T$ satisfies the Condition \ $I,$
that is, $\rho (f-Tf)\geq \ell (dist_{\rho }(f,F_{\rho }(T)))$ for all $f\in
D.$

Indeed, if $f\in F_{\rho }(T)=\{1\},$ then\ obviously 
\begin{equation*}
\rho (f-Tf)=0=\ell (dist_{\rho }(f,F_{\rho }(T))).
\end{equation*}%
If $f\in (1,\infty ),$ then%
\begin{eqnarray*}
\rho (f-Tf) &=&\rho \left( f-\left( \frac{2f+1}{3}\right) \right) \\
&=&\left\vert f-\left( \frac{2f+1}{3}\right) \right\vert \\
&=&\frac{f-1}{3},
\end{eqnarray*}%
and%
\begin{equation*}
\ell (dist_{\rho }(f,F_{\rho }(T)))=\ell (dist_{\rho }(f,\{1\}))=\ell
(\left\vert f-1\right\vert )=\frac{f-1}{6}.
\end{equation*}%
Thus $\rho (f-Tf)\geq \ell (dist_{\rho }(f,F_{\rho }(T)))$ for all $f\in D.$
We next show that $T$ is $\left( \lambda ,\rho \right) $-firmly
nonexpansive. Fix $\lambda =\frac{1}{3}.$Then 
\begin{eqnarray*}
\rho (Tf-Tg) &=&\left\vert Tf-Tg\right\vert \\
&=&\left\vert \frac{2f+1}{3}-\frac{2g+1}{3}\right\vert \\
&=&\frac{2}{3}\left\vert f-g\right\vert \\
&\leq &\frac{8}{9}\left\vert f-g\right\vert \\
&=&\left\vert \frac{2}{3}(f-g)+\frac{1}{3}\left[ \frac{2}{3}(f-g)\right]
\right\vert \\
&=&\left\vert \frac{2}{3}(f-g)+\frac{1}{3}(Tf-Tg)\right\vert \\
&=&\rho \left( \frac{2}{3}(f-g)+\frac{1}{3}(Tf-Tg)\right) .
\end{eqnarray*}%
Thus $T$ is $\left( \lambda ,\rho \right) $-firmly nonexpansive. Lastly, we
show that $\left\{ f_{n}\right\} $ $\rho $-converges to $1,$ the fixed point
of $T.$ For this, fix the starting point of the algorithm as $f_{1}=4$ and
choose $\alpha _{n}=\frac{1}{2}$ for all $n\in \mathbb{N}$ for simplicity.
Then $Tf_{n}=(2f_{n}+1)/3$ and $g_{n}=(0.5)\left( f_{n}+Tf_{n}\right) .$
\end{example}

\begin{tabular}{|l|l|l|l|l|}
\hline
$n$ & $f_{n}$ & $Tf_{n}$ & $g_{n}$ & $f_{n+1}=Tg_{n}$ \\ \hline
1 & 4.000000 & 3.000000 & 3.500000 & 2.666667 \\ \hline
2 & 2.666667 & 2.111111 & 2.388889 & 1.925926 \\ \hline
3 & 1.925926\qquad \qquad \qquad & 1.617284 & 1.771605 & 1.514403 \\ \hline
4 & 1.514403 & 1.342936\qquad \qquad & 1.428669 & 1.285780 \\ \hline
5 & 1.285780 & 1.190520\qquad \qquad & 1.238150 & 1.158766 \\ \hline
10 & 1.015124 & 1.010083 & 1.012603 & 1.008402 \\ \hline
15 & 1.000800\qquad \qquad \qquad & 1.000534 & 1.000667 & 1.000445 \\ \hline
20 & 1.000042\qquad \qquad \qquad & 1.000028 & 1.000035 & 1.000024 \\ \hline
22 & 1.000013\qquad \qquad \qquad & 1.000009 & 1.000011 & 1.000007 \\ \hline
\end{tabular}

\textit{The above table, created by using Microsoft Excel, shows that }$%
\left\{ f_{n}\right\} $\textit{\ }$\rho $\textit{-converges to }$1,$\textit{%
\ the fixed point of }$T,$\textit{\ to the accuracy of }$10^{-5}$\textit{\
on }$22$\textit{nd iteration. On further{}computations, the accuracy
increases to }$10^{-10}$\textit{\ on }$42$\textit{nd iteration.}

\begin{remark}
In the above example, $\left\{ f_{n}\right\} $ $\rho $-converges\ faster to $%
1$ if we take $\alpha _{n}$ near the fixed point. For example, if we take $%
\alpha _{n}=0.75,$ then the convergence to the accuracy of $10^{-5}$ is
obtained on $19$th iteration. But if we take $\alpha _{n}=0.25,$ far from $%
1, $ the required accuracy is achieved on $26$th iteration.
\end{remark}

\section{Concluding Remarks}

We have proved some strong convergence results using $\left( \lambda ,\rho
\right) $-firmly nonexpansive mappings on a faster iterative algorithm in
modular function spaces. In our opinion it would be interesting to consider
the following using above ideas:

(1) studying the stabiltiy and data dependency problems

(2) finding applications to general variational inequalities or equilibrium
problems as well as

to split feasibility problems.

We may suggest the redaer to combine the ideas studied in \cite{TTP1, TTP2,
YPLY,YAPL,YLP,YLPZ}

\end{document}